\documentclass[12pt,a4paper]{amsart}
\usepackage[top=1.15in, bottom=1.15in, left=1.17in, right=1.17in]{geometry}
\usepackage[T1]{fontenc}
\usepackage[utf8]{inputenc}
\usepackage{lmodern}
\usepackage{microtype}
\usepackage{ellipsis}
\usepackage{amsmath}
\usepackage{amssymb}
\usepackage{enumitem}
\usepackage[subrefformat=parens]{subcaption}
\usepackage[english]{babel}

\usepackage[breaklinks,colorlinks,cite color=blue,pdfusetitle]{hyperref}
\usepackage{amsthm}
\usepackage{cleveref}
\usepackage{thmtools}

\declaretheorem[numberwithin=section]{theorem}
\declaretheorem[numbered=no,name=Theorem]{theorem*}
\declaretheorem[sibling=theorem]{proposition}
\declaretheorem[sibling=theorem]{lemma}

\declaretheorem[sibling=theorem]{corollary}
\declaretheorem[style=definition,sibling=theorem]{definition}
\declaretheorem[style=definition,sibling=theorem]{remark}

\declaretheorem[style=definition,sibling=theorem,qed=$\lozenge$]{example}

\usepackage{mathtools}
\usepackage{tikz}
\usetikzlibrary{cd}
\usepackage[backend=biber,maxbibnames=99]{biblatex}
\numberwithin{equation}{section}

\newcommand\OO{\mathcal O}
\newcommand\CC{\mathcal C}
\newcommand\II{\mathcal I}

\newcommand\LL{\mathcal L}

\newcommand\R{\mathbb R}
\newcommand\C{\mathbb C}
\newcommand\Z{\mathbb Z}
\newcommand\N{\mathbb N}
\newcommand\Q{\mathbb Q}

\newcommand{\br}{\mathbf{r}}

\newcommand{\bv}{\mathbf{v}}

\newcommand{\bx}{\mathbf{x}}
\newcommand{\by}{\mathbf{y}}
\newcommand{\bz}{\mathbf{z}}

\newcommand{\bu}{\mathbf{u}}

\newcommand{\g}{\mathfrak{g}}
\newcommand{\ve}{\varepsilon}

\newcommand{\la}{\lambda}

\newcommand{\s}{\sigma}

\newcommand{\om}{\omega}

\DeclareMathOperator\aff{aff}
\DeclareMathOperator\conv{conv}

\DeclareMathOperator\im{im}

\DeclareMathOperator\inte{int}

\DeclareMathOperator\SL{SL}
\DeclareMathOperator\D{D}

\DeclareMathOperator\G{G}

\DeclareMathOperator\supp{supp}

\tikzset{posetelm/.style={draw, fill, circle, minimum size=4pt, inner sep=0}}
\tikzset{posetelmm/.style={draw, thick, minimum size=5pt, inner sep=0}}
\tikzset{marking/.style={red}}
\tikzset{elmname/.style={blue}}
\tikzset{covrel/.style={thick}}

\title[Minkowski property and reflexivity]{The Minkowski Property and Reflexivity of Marked Poset Polytopes}

\author[X.\ Fang, G.\ Fourier, C.\ Pegel]{Xin Fang, Ghislain Fourier, Christoph Pegel}

\address{Mathematical Institute, University of Cologne, Germany}
\email{xinfang.math@gmail.com}
\address{Lehrstuhl B für Mathematik, RWTH Aachen University, Germany}
\email{fourier@math.uni-hannover.de}
\address{Institute for Algebra, Number Theory, and Discrete Mathematics, Leibniz University Hannover, Germany}
\email{pegel@math.uni-hannover.de}

\begin{document}
\begin{abstract}
We provide a Minkowski sum decomposition of marked chain-order polytopes into building blocks associated to elementary markings and thus give an explicit minimal set of generators of an associated semi-group algebra. We proceed by characterizing the reflexive polytopes among marked chain-order polytopes as those with the underlying marked poset being ranked.
\end{abstract}

\maketitle

\section*{Introduction}
To a given finite poset, Stanley \cite{Sta86} associated two polytopes---the order polytope and the chain polytope, which are lattice polytopes having the same Ehrhart polynomial. When the underlying poset is a distributive lattice, Hibi studied in \cite{Hibi87} the geometry of the toric variety associated to the order polytope, nowadays called Hibi varieties. Together with Li, they also initiated the study of the toric variety associated to the chain polytope \cite{HL15}. The singularities of Hibi varieties arising from Gelfand-Tsetlin degenerations of Grassmann varieties are studied by Brown and Lakshmibai in \cite{BL10} (see also the references therein).

Motivated by the representation theory of complex semi-simple Lie algebras, namely the framework of PBW-degenerations, Ardila, Bliem and Salazar \cite{ABS11} introduced the notion of marked order polytopes and marked chain polytopes, defined on marked posets. They showed that they are lattice polytopes, and for a fixed marked poset, they share the same Ehrhart polynomial. Their motivating example is the Gelfand-Tsetlin polytopes and the Feigin-Fourier-Littelmann-Vinberg (FFLV) polytopes \cite{FFL11}, which are respectively marked order polytopes and marked chain polytopes associated to particular marked posets (which are in fact distributive lattices). The toric varieties associated to these polytopes can be obtained from toric degenerations of flag varieties of type $\tt A$ (\cite{GL96, MS, FFL17}). The geometric properties of the toric varieties associated to Gelfand-Tsetlin polytopes are investigated in \cite{BCKS}. Some vertices of the FFLV polytopes are studied in \cite{FM17}. 

To put these two families of polytopes into a continuous family, the third author introduced a one-parameter family interpolating the marked order and the marked chain polytopes continuously.

Motivated by the work on linear degenerations of flag varieties \cite{CFFFR17}, seeking for intermediate lattice polytopes between the marked order and the marked chain polytopes, as well as toric degenerations of the linear degenerate flag varieties to these polytopes, becomes a meaningful question.

A first step in this direction has been initiated by the first two authors in \cite{FF16}, where such polytopes are defined under certain restrictions. Later in the joint work with J.-P.\ Litza \cite{FFLP17}, this approach is combined with the interpolating one parameter family: we introduced a family interpolating the marked order and the marked chain polytopes, called marked poset polytopes, parametrized by points in a hypercube. In this family, every vertex of the hypercube corresponds to a lattice polytope (called a marked chain-order polytope) and they all have the same Ehrhart polynomial.

The goal of this paper is to study the algebro-geometric properties of the toric varieties associated to the marked chain-order polytopes.

The first result of this paper is a decomposition of marked chain-order polytopes into Minkowski sums of building blocks associated to elementary markings (\Cref{Thm:CO-decomp,Thm:CO-cone-decomp}). As a consequence, we provide an explicit set of minimal generators (of homogeneous degree one) of the semi-group algebra of the associated toric variety.

Reflexive polytopes arise from the study of mirror symmetry for Calabi-Yau hypersurfaces in toric varieties \cite{Bat94}. Geometrically, the toric variety associated to a reflexive polytope is Gorenstein--Fano. The second main result of this paper is concerned with characterizing reflexive polytopes in arbitrary dimensions from marked chain-order polytopes. Indeed, for any ranked marked poset and any vertex in the parametrizing hypercube, we construct a reflexive polytope (many of them are not unimodular equivalent) by choosing a proper marking (\Cref{Thm:main}).

The paper is structured as follows: After recalling the construction of marked chain-order polytopes and basic facts on reflexive polytopes in \Cref{Sec:1}, we study the Minkowski decomposition property in \Cref{sec:minkowski} and the reflexivity in \Cref{Sec:3}.

\section{Preliminaries}\label{Sec:1}

The set $\R$ (resp. $\Q$, $\Z$, $\N$) of real (resp. rational, integral, natural) numbers is endowed with the usual total order.

For a polytope $Q\subseteq\R^N$, we denote by $Q^\Z= Q\cap\Z^N$ the set of lattice points in $Q$. For $c\in\R$ we denote by $c\,Q$ the dilation of $Q$ consisting of all points $c\,\bx$ for $\bx\in Q$. When $c$ is a positive integer, we denote by $c\, Q^\Z$ the Minkowski sum of $c$ copies of $Q^\Z$, by $(-c)\,Q^\Z$ the Minkowski sum of $c$ copies of $-Q$ and by $0\,Q^\Z$ the set $\{\mathbf 0\}\subseteq\R^N$.

\subsection{Marked Chain-Order Polytopes}

Let \((P, {\le})\) be a finite poset and denote covering relations by $p\prec q$ \cite{Sta11}.
We first recall the notion of a marked poset \cite{ABS11}.

\begin{definition}\label{def:abs-mpp}
A pair \((P,\la)\) is called a \emph{marked poset}, if \(P^*\subseteq P\) is an induced subposet and \(\la\colon P^*\to\R\) is an order-preserving map on $P^*$. The map $\la$ is called a \emph{marking}; members in $P^*$ are called \emph{marked elements}. We denote by \({\tilde P} = P\setminus P^*\) the set of all \emph{unmarked elements}. The marking $\la$ is called \emph{integral}, if $\im(\la)\subseteq\Z$.
\end{definition}

In this paper, we will assume throughout that at least all minimal elements in $P$ are marked: $\min(P)\subseteq P^*$.
To a marked poset we associated in \cite{FFLP17} a family of polyhedra parametrized by partitions $\tilde P=C\sqcup O$. 

\begin{definition} \label{def:mcop}
    Let $P = P^*\sqcup C\sqcup O$ be a partition of a poset $P$ with $\min(P)\subseteq P^*$ and $\la$ a marking. The elements of $C$ and $O$ are called \emph{chain elements} and \emph{order elements}, respectively.
    The \emph{marked chain-order polyhedron} $\OO_{C,O}(P,\la)\subseteq\R^{P}$ is the set of all $\bx=(x_p)_{p\in P}\in\R^P$ satisfying the following conditions:
    \begin{enumerate}
        \item for any $a\in P^*$, $x_a=\la(a)$;
        \item for $p\in C$, $x_p\geq 0$;
        \item for each saturated chain $a\prec p_1\prec\cdots\prec p_r\prec b$ with $a,b\in P^*\sqcup O$, $p_i\in C$, $r\geq 0$, we have
            \[x_{p_1}+\cdots+x_{p_r}\leq x_b-x_a.\]
    \end{enumerate}
    \end{definition}
    When a partition $P=P^*\sqcup C\sqcup O$ is given, we write the points of $\R^P$ as $\bx=(\la,\bx_C,\bx_O)$ with $\la\in\R^{P^*}$, $\bx_C\in\R^C$ and $\bx_O\in\R^O$.
    Since the coordinates in $P^*$ are fixed for the points of $\OO_{C,O}(P,\la)$, we sometimes consider the projection of $\OO_{C,O}(P,\la)$ in $\R^{\tilde P}$ instead, keeping the same notation to write $(\bx_C,\bx_O)\in\OO_{C,O}(P,\la)$ instead of $(\la,\bx_C,\bx_O)\in\OO_{C,O}(P,\la)$.

\begin{remark}
\begin{enumerate}
\item When $\min(P)\cup\max(P)\subseteq P^*$ and $\la$ is integral, the marked chain-order polyhedron $\OO_{C,O}(P,\la)$ is a lattice polytope. In this case, the notion of marked chain-order polytope is different from the one in \cite{FF16}, where the assumption that $C$ is an order ideal was made.
\item When in addition $C=\varnothing$ the marked chain-order polytope $\OO_{C,O}(P,\la)$ is the marked order polytope and will be denoted by $\OO(P,\la)$; when $O=\varnothing$ the marked chain-order polytope $\OO_{C,O}(P,\la)$ is the marked chain polytope and will be denoted by $\mathcal{C}(P,\la)$ (for the marked order and the marked chain polytopes, see \cite{ABS11}).
\end{enumerate}
\end{remark}

\subsection{Regular and Ranked Marked Posets}

We recall two important properties of marked posets: the regularity and the rankedness.

\begin{definition}
\begin{enumerate}
\item (\cite{Peg17}) A marked poset $(P,\la)$ is called \emph{regular} if for each covering relation $p\prec q$ in $P$ and $a,b\in P^*$ such that $a\leq q$ and $p\leq b$, we have $a=b$ or $\la(a)<\la(b)$.
\item A poset $P$ is called \emph{ranked}, if there exists a rank function $r\colon P\to\Z$ satisfying: for each covering relation $p\prec q$ in $P$, $r(p)=r(q)-1$.
\item (\cite{FFLP17}) A marked poset $(P,\la)$ is called \emph{ranked}, if $P$ is a ranked poset with rank function $r$ such that for any $a,b\in P^*$ with $r(a)<r(b)$, we have $\la(a)<\la(b)$.
\end{enumerate}
\end{definition}

Let $P$ be a ranked poset and $P^*\subseteq P$ a set of marked elements. Any rank function $r$ defines a marking $\la^r\colon P^*\to\Z$ by letting $\la^r(a)=r(a)$ for all $a\in P^*$.

The following lemma is clear by definition.

\begin{lemma}\label{Lem:RankReguar}
The marked poset $(P,\la^r)$ is regular. \qed
\end{lemma}

For a regular marked poset, the facets of the marked order polytope are given in \cite{Peg17}.

\begin{proposition}\label{Prop:Facet}
    Let $(P,\la)$ be a regular marked poset. The facet-defining inequalities of $\OO(P,\la)$ are $x_p\le x_q$ for all $p,q\in P$ with $p\prec q$.\footnotemark For the projection of $\OO(P,\la)$ in $\R^{\tilde P}$, the facets are expressed as:
\begin{enumerate}
\item for a covering relation $p\prec q$ in $\tilde{P}$, $x_p\leq x_q$;
\item for $a\in P^*$, $p\in\tilde{P}$ such that $a\prec p$, $\la(a)\leq x_p$;
\item for $b\in P^*$, $q\in\tilde{P}$ such that $q\prec b$, $x_q\leq \la(b)$. \qed
\end{enumerate}
\end{proposition}
\footnotetext{Note that there are no covering relations between marked elements in a regular marked posets \cite[Proposition~3.20]{Peg17}.}

\subsection{Ehrhart-Macdonald Reciprocity}

Let $Q\subseteq\R^n$ be a lattice polytope. The Ehrhart polynomial $L_Q(x)\in\Q[x]$ is a polynomial of degree $\dim(Q)$, satisfying for all $m\in\N$:
\[L_Q(m)=\# (mQ\cap\Z^n).\]
By the famous \emph{Ehrhart-Macdonald reciprocity} \cite{Mac71, BR15} of the Ehrhart polynomial, for all $m\in\N$ we have
\[L_Q(-m)=(-1)^{\dim(Q)}\#(\inte(mQ)\cap\Z^n).\]

Two polytopes are called \emph{Ehrhart equivalent}, if they have the same Ehrhart polynomial.

The following result on the Ehrhart equivalence of the marked chain-order polytopes is proved in \cite{FFLP17} with the help of a transfer map.
\begin{proposition}\label{Prop:Ehrhart}
Let $(P,\la)$ be an integrally marked poset such that $\min(P)\cup\max(P)\subseteq P^*$. For any two partitions $\tilde{P}= C\sqcup O= C'\sqcup O'$, the polytopes $\OO_{C,O}(P,\la)$ and $\OO_{C',O'}(P,\la)$ are Ehrhart equivalent. \qed
\end{proposition}

\subsection{Reflexive Polytopes}
Let $Q\subseteq \R^n$ be a polytope. The polar dual of $Q$ is defined to be:
\[Q^\circ=\{\alpha\in(\R^n)^*\mid \alpha(\mathbf{x})\leq 1\text{ for all }\mathbf{x}\in Q\}.\]

Let $Q\subseteq \R^n$ be a polytope with $0\in\inte(Q)$, then the polar dual $Q^\circ$ is a polytope. 
A polytope $Q$ with $0\in\inte(Q)$ is called \emph{reflexive} (see \cite{Bat94}), if both $Q$ and $Q^\circ$ are lattice polytopes. Toric varieties associated to reflexive polytopes are Gorenstein-Fano, which play an important role in Batyrev's construction in mirror symmetry.

Hibi gives \cite{Hib92} a beautiful criterion on the integrality of the dual of a rational polytope. For our application, we recall it in the case of lattice polytopes.

\begin{theorem}\label{Thm:Hibi}
Let $Q\subseteq\R^n$ be a lattice polytope with $0\in\inte(Q)$. Then $Q^\circ$ is a lattice polytope if and only if for all $m\in\N$,
\pushQED{\qed}
\begin{equation*}
L_Q(-m-1)=(-1)^nL_Q(m). \qedhere
\end{equation*}
\popQED
\end{theorem}

In particular, if $Q^\circ$ is a lattice polytope, by taking $m=0$, the above theorem and the Ehrhart-Macdonald reciprocity imply that $0$ is the only interior lattice point in $Q$.

We finish this subsection by recalling the polarity theorem \cite{Zie95}, which will be used later.

\begin{proposition}\label{Prop:Polarity}
Let $Q\subseteq\R^n$ be a polytope with $0\in\inte(Q)$. Assume that 
\[Q=\conv\{\bv_1,\dots,\bv_s\}=\{\bx\in\R^n\mid \alpha_1(\bx)\leq 1,\dots,\alpha_t(\bx)\leq 1\}\]
be the descriptions of $Q$ by its vertices and facets where $\alpha_1,\dots,\alpha_t\in(\R^n)^*$. Then the polar dual
\[Q^\circ=\conv\{\alpha_1,\dots,\alpha_t\}=\{\alpha\in (\R^n)^*\mid \alpha(\bv_1)\leq 1,\dots,\alpha(\bv_s)\leq 1\}\]
is the description of $Q^\circ$ by its vertices and facets. \qed
\end{proposition}

The polarity theorem implies the following geometric characterization of reflexive polytopes.

\begin{theorem}\label{Thm:GeomReflexive}
Assume that $Q$ is a lattice polytope with $0\in\inte(Q)$. Then $Q$ is reflexive if and only if for each of its facets $F$, there is no lattice point between the hyperplane $\aff(F)$ and its parallel through the origin. \qed
\end{theorem}

\section{Minkowski Property of Marked Chain-Order Polytopes}\label{sec:minkowski}

The goal of this section is to give a decomposition of marked chain-order polyhedra as Minkowski sums of marked chain-order polyhedra that are given by zero-one markings. For marked order polytopes, this Minkowski decomposition property has been discussed in \cite{JS14,SP02}. For marked chain polytopes it appeared in \cite{Fou16}. We generalize the results to all marked chain-order polyhedra.

\subsection{Minkowski Decomposition Property}

\begin{lemma}
    \label{lem:minkowski-inclusion}
    Let $\la,\mu\colon P^*\to \R$ be two markings on a poset $P$ and let $\tilde P = C\sqcup O$ be any partition. We have
    \begin{align*}
        \OO_{C,O}(P,\la)+\OO_{C,O}(P,\mu) &\subseteq \OO_{C,O}(P,\la+\mu), \\
        \OO_{C,O}^\Z(P,\la) +\OO_{C,O}^\Z(P,\mu) &\subseteq \OO_{C,O}^\Z(P,\la+\mu).
    \end{align*}
\end{lemma}
\begin{proof}
    This result is immediate by summing the defining inequalities of $\OO_{C,O}(P,\la)$ and $\OO_{C,O}(P,\mu)$ from \Cref{def:mcop}. For example, when $\bx\in\OO_{C,O}(P,\la)$, $\by\in\OO_{C,O}(P,\mu)$ and $a\prec p_1\prec\cdots\prec p_r\prec b$ is one of the defining saturated chains, we have
    \begin{align*}
        x_{p_1} + \cdots + x_{p_r} &\le x_b - x_a \quad\text{and}\\
        y_{p_1} + \cdots + y_{p_r} &\le y_b - y_a \\
        \intertext{and thus}
        (x_{p_1} + y_{p_1}) + \cdots + (x_{p_r} + y_{p_r}) &\le (x_b + y_b) - (x_a + y_a),
    \end{align*}
    which is one of the conditions for $\bx+\by\in\OO_{C,O}(P,\la+\mu)$.
\end{proof}

Note that the other inclusion ``$\supseteq$'' does not hold in general:

\begin{example}
    \label{ex:bad-decomp}
    Consider the following marked poset:
    
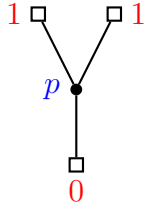
\begin{align*}
    \begin{tikzpicture}[scale=1]
        \path (.5,1) node[posetelm] (p) {} node[left=2pt,elmname] {\(p\)};
        \path (.5,0) node[posetelmm] (0) {} node[below=2pt,marking] {\(0\)};
        \path (0,2) node[posetelmm] (1) {} node[left=2pt,marking] {\(1\)};
        \path (1,2) node[posetelmm] (1b) {} node[right=2pt,marking] {\(1\)};
        \draw[covrel]
              (0) -- (p) -- (1) 
             (p) -- (1b);
         \end{tikzpicture}
    \end{align*}
    The associated marked order polytope is a line segment. Decomposing the marking as
    \begin{align*}
        (0,1,1) = (0,1,0) + (0,0,1),
    \end{align*}
    we see that the marked order polytopes associated to the summands are both just a point at the origin.
    Hence, their sum is not the original line segment.
\end{example}

\begin{definition}
    \label{def:elementary}
    A marking $\om$ of a marked poset $(P,\om)$ is called \emph{elementary}, if it is the indicator function $\chi_F\colon P^*\to\{0,1\}$ of some filter $F\subseteq P^*$.
\end{definition}

The marking in the above example is elementary, since it is a zero-one marking with the set of elements marked one being upward closed.

\begin{definition}
    \label{def:marking-decomp}
    Let $(P,\la)$ be any marked poset and $\la(P^*) = \{ c_0 < c_1 < \dots < c_k \}$.
    To $\la$ we associate the elementary markings $\om_i\colon P^*\to \R$ for $i=0,\dots,k$ given by $\om_i=\chi_{F_i}$ for the filters
    \begin{align*}
        F_i = \la^{-1}(\R_{\ge c_i}).
    \end{align*}
    That is, $\om_i(a)$ is $1$ if $\la(a)\ge c_i$ and $0$ otherwise.
    We refer to the decomposition
    \begin{align*}
        \la = c_0 \om_0 + (c_1-c_0) \om_1 + \cdots + (c_k-c_{k-1}) \om_k
    \end{align*}
    as the \emph{decomposition of $\la$ into elementary markings}.
\end{definition}

\begin{proposition}
    \label{prop:order-decomp}
    The marked order polyhedron $\OO(P,\la)$ decomposes as the weighted Minkowski sum
    \begin{align*}
        \OO(P,\la) = c_0 \, \OO(P,\om_0) + (c_1-c_0) \, \OO(P,\om_1) + \cdots + (c_k-c_{k-1}) \, \OO(P,\om_k).
    \end{align*}
    Furthermore, when $\la$ is integral, we have
    \begin{align*}
        \OO^\Z(P,\la) = c_0 \, \OO^\Z(P,\om_0) + (c_1-c_0) \, \OO^\Z(P,\om_1) + \cdots + (c_k-c_{k-1}) \, \OO^\Z(P,\om_k).
    \end{align*}
\end{proposition}
\begin{proof}
    First note that $\om_0$ is the constant marking of all ones. Hence, the all-one vector $\mathbf 1\in \R^P$ is a lattice point in $\OO(P,\om_0)$. Now let $\bx$ be any point in $\OO(P,\la)$, then by subtracting $c_0$ from all coordinates, we see that $\bx-c_0 \mathbf 1\in \OO(P,\la-c_0 \om_0)$. This implies that
    \begin{equation*}
        \OO(P,\la) = \OO(P,\la-c_0 \om_0) + c_0\,\OO(P,\om_0).
    \end{equation*}
    Note that, when $c_0\neq0$, $c_0\,\OO(P,\om_0)$ is just the recession cone of $\OO(P,\la)$ shifted by $c_0 \mathbf 1$ and $\OO(P,\la-c_0\om_0)$ is just $\OO(P,\la)$ shifted by $-c_0\mathbf 1$.
    When $\la$ and $\bx$ are both integral, the above construction yields
    \begin{equation*}
        \OO^\Z(P,\la) = \OO^\Z(P,\la-c_0 \om_0) + c_0\,\OO^\Z(P,\om_0).
    \end{equation*}
    
    We may thus assume that $c_0=0$ and all $c_i\ge 0$, replacing $\la$ by $\la-c_0\om_0$ otherwise. Given $c_0=0$, we now show for any $0<\ve\le c_1$ that
    \begin{equation*}
        \OO(P,\la) = \OO(P,\la-\ve \om_1) + \ve\,\OO(P,\om_1).
    \end{equation*}
    Let $\bx\in\OO(P,\la)$ and define $\by\in\R^P$ by $y_p = \min\{x_p, \ve\}$.
    Note that for $p\in P^*$ we have $\ve \om_1(p)=\min\{\la(p),\ve\}$ and hence $\by\in\OO(P,\ve\om_1)=\ve\OO(P,\om_1)$ since $a\le b$ implies $\min\{a,c\}\le\min\{b,c\}$. Let $\bz=\bx-\by$ so that $z_p=x_p-\min\{x_p,\ve\}=\max\{0,x_p-\ve\}$. We have $\bz\in\OO(P,\la-\ve\om_1)$, since $a\le b$ implies $\max\{0,a-c\}\le\max\{0,b-c\}$.
    Choosing $\ve=c_1$ we have
    \begin{equation*}
        \OO(P,\la) = \OO(P,\la-c_1 \om_1) + c_1\,\OO(P,\om_1).
    \end{equation*}
    When $\la$ and $\bx$ are integral, we may choose $\ve=1$ so that the points $\by$ and $\bz$ are integral as well and we have
    \begin{equation*}
        \OO^\Z(P,\la) = \OO^\Z(P,\la-\om_1) + \OO^\Z(P,\om_1).
    \end{equation*}
    Applying this decomposition $c_1$ times, we conclude
    \begin{equation*}
        \OO^\Z(P,\la) = \OO^\Z(P,\la-c_1 \om_1) + c_1\,\OO^\Z(P,\om_1).
    \end{equation*}

    Inductively repeating the above procedure yields the Minkowski decomposition given in the statement of the proposition: after subtracting $c_0\om_0+c_1 \om_1$ from the marking, we changed the marking $\la$ to a marking $\la'$ with $c_0$ and $c_1$ replaced by $0$ and $c_2,\dots,c_k$ replaced by $c_2-c_1-c_0,\dots,c_k-c_1-c_0$. Now the elementary marking $\om'_1$ for $\la'$ is exactly the elementary marking $\om_2$ for the original $\la$.
\end{proof}

The Minkowski decomposition of $\OO(P,\la)$ already appeared in \cite{JS14, SP02} for the case of bounded polyhedra and $P^*$ being a chain in $P$.
For arbitrary marked order polyhedra it was shown in \cite{Peg17} using different techniques.
Note that the unbounded case could as well be obtained from the bounded case by considering the maximal elements in $P$ marked when decomposing individual points.
However, the inductive approach taken here (subtracting an $\varepsilon$-amount of $\omega_1$ in each step) becomes important in the generalization to marked chain-order polyhedra and in obtaining integral decomposition of lattice points.

The following proposition generalizes the Minkowski property of marked chain polytopes proved in \cite{Fou16} for the case of $P^*$ being a chain.
 
\begin{proposition}
    \label{prop:chain-decomp}
    The marked chain polyhedron $\CC(P,\la)$ decomposes as the weighted Minkowski sum
    \begin{align*}
        \CC(P,\la) = c_0 \, \CC(P,\om_0) + (c_1-c_0) \, \CC(P,\om_1) + \cdots + (c_k-c_{k-1}) \, \CC(P,\om_k).
    \end{align*}
    Furthermore, when $\la$ is integral, we have
    \begin{align*}
        \CC^\Z(P,\la) = c_0 \, \CC^\Z(P,\om_0) + (c_1-c_0) \, \CC^\Z(P,\om_1) + \cdots + (c_k-c_{k-1}) \, \CC^\Z(P,\om_k).
    \end{align*}
\end{proposition}
\begin{proof}
    Since $\om_0$ is the constant marking of all ones, the marked chain polyhedron $\CC(P,\om_0)$ is the recession cone of $\CC(P,\la)$ and $\CC(P,\la-c_0\om_0)$ is in fact equal to $\CC(P,\la)$. To see this, note that all defining inequalities involving the marking are of the form $x_{p_1}+\cdots+x_{p_r} \le \la(b)-\la(a)$, so that changing all markings to $1$ yields the defining inequalities of the recession cone while subtracting $c_0$ from all markings does not change these inequalities at all. Hence, we have
    \begin{equation*}
        \CC(P,\la) = \CC(P,\la-c_0\om_0) + c_0\, \CC(P,\om_0).
    \end{equation*}
    Furthermore, when $\la$ is integral, we may decompose lattice points in $\CC(P,\la)$ as $(\la,\bx_C)=(\la-c_0\om_0,\bx_C)+c_0\,(\om_0,\mathbf 0)$, which yields
    \begin{equation*}
        \CC^\Z(P,\la) = \CC^\Z(P,\la-c_0\om_0) + c_0\, \CC^\Z(P,\om_0).
    \end{equation*}

    Thus, we may assume $c_0=0$ as before. To show that
    \begin{equation*}
        \CC(P,\la) = \CC(P,\la-c_1 \om_1) + c_1 \, \CC(P,\om_1),
    \end{equation*}
    we take $\bx\in\CC(P,\la)$ and define $\by\in\R^P$ in the following way: denote by $S$ the set of all $p\in C$ such that there is no $a<p$ with $\om_1(a)=1$ and no $b>p$ with $\om_1(b)=0$. Denote by $\supp(\bx)$ the set of all $p\in C$ such that $x_p>0$.
    When $S\cap\supp(\bx)$ is empty, we may just decompose $\bx$ as
    \begin{align*}
        (\la, \bx_C) = (\la-c_1\om_1, \bx_C) + (c_1\om_1, \mathbf 0).
    \end{align*}
    Otherwise, let $\tilde\ve>0$ be the minimum over all $x_p$ for $p$ a minimal element in $S\cap\supp(\bx)$ as an induced subposet of $P$ and set $\ve=\min\{\tilde\ve,c_1\}$.
    Now define $\by\in\R^P$ by letting
    \begin{equation*}
        y_p =
        \begin{cases}
            \ve\om_1(p) & \text{for $p\in P^*$,} \\
            \ve & \text{for $p\in \min(S\cap\supp(\bx))$}, \\
            0 & \text{otherwise,}
        \end{cases}
    \end{equation*}

    We claim that $\by\in\CC(P,\ve \om_1)$. We have $y_p\ge 0$ for all $p\in C$ by definition. Now consider any saturated chain $a\prec p_1 \prec \cdots \prec p_r \prec b$ with $a,b\in P^*$ and all $p_i\in C$. We have to verify
    \begin{equation*}
        y_{p_1}+\cdots+y_{p_r} \le \ve (\om_1(b)-\om_1(a)).
    \end{equation*}
    When $\om_1(a)=1$ or $\om_1(b)=0$, both sides of the inequality are zero: the left hand side is zero since none of the $p_i$ are elements of $S$, the right hand side is zero since $\om_1(a)=\om_1(b)$ in this case. When $\om_1(a)=0$ and $\om_1(b)=1$, at most one of the $p_i$ is a minimal element in $S\cap\supp(\bx)$, since the $p_i$ are elements of a chain. Hence, the left hand side is at most $\ve$ in this case and we conclude that $\by\in\CC(P,\ve\om_1)$.

    Now consider $\bz=\bx-\by$. The coordinates of $\bz$ are
    \begin{equation*}
        z_p =
        \begin{cases}
            \la(p) - \ve \om_1(p) & \text{for $p\in P^*$,} \\
            x_p - \ve & \text{for $p\in \min(S\cap\supp(\bx))$}, \\
            x_p & \text{otherwise.}
        \end{cases}
    \end{equation*}   

    We have $z_p\ge 0$ for all $p\in C$, since we only subtract $\ve$ for coordinates in $\min(S\cap\supp(x))$. Now consider any saturated chain $a\prec p_1 \prec \cdots \prec p_r \prec b$ with $a,b\in P^*$ and all $p_i\in C$. Since $\bx\in\CC(P,\la)$, we have
    \begin{equation}
        x_{p_1}+\cdots+x_{p_r} \le \la(b)-\la(a).
        \label{eq:given}
    \end{equation}
    The corresponding condition for $\bz\in\CC(P,\la-\ve \om_1)$ is
    \begin{equation}
        z_{p_1}+\cdots+z_{p_r} \le (\la(b)-\ve \om_1(b))-(\la(a)-\ve \om_1(a)). \label{eq:check}
    \end{equation}
    We only have to verify this for cases where the right hand side of \eqref{eq:given} got decreased by $\ve$ in \eqref{eq:check}, i.e., when $\om_1(a)=0$ and $\om_1(b)=1$. In all other cases, the right hand side of \eqref{eq:check} is the same as in \eqref{eq:given} while the left hand side possibly decreased by $\ve$.

    Hence, we assume $\om_1(a)=0$ and $\om_1(b)=1$, so we have $\la(a)=0$ and $\la(b)\ge c_1 \ge \ve > 0$.
    If all $x_{p_i}$ are zero, all $z_{p_i}$ are zero as well and \eqref{eq:check} is trivially satisfied. Otherwise, let $j$ be the smallest index such that $x_{p_j}>0$. If there is some $a'<p_j$ with $\om_1(a')=1$, the chain $a'<p_j\prec \cdots \prec p_r\prec b$ yields%
    \footnote{This chain is not saturated and hence does not correspond to a defining inequality of $\CC(P,\la)$ by our definition. However, the chain may be refined to a saturated one and split at every marked element to obtain the given inequality.}
    \begin{equation*}
        \underbrace{x_{p_1} + \cdots + x_{p_{j-1}}}_0 + x_{p_j} + \cdots +x_{p_r} \le
        \la(b) - \la(a') \le \la(b) - \ve = \la(b)-\la(a)-\ve,
    \end{equation*}
    since $\la(a')\ge c_1\ge \ve > 0 = \la(a)$.
    Hence, decreasing the right hand side of \eqref{eq:given} by $\ve$ still yields a valid inequality, regardless of the left hand side being decreased or not.

    If there is some $b'>p_j$ with $\om_1(b')=0$, the chain $a<p_j<b'$ yields
    \begin{equation*}
        x_{p_j} \le \la(b')-\la(a) = 0-0 = 0,
    \end{equation*}
    which contradicts the choice of $p_j$.

    Hence, we may assume that $p_j\in S\cap\supp(\bx)$. If $p_j$ is a minimum in $S\cap\supp(\bx)$, the left hand side decreases by $\ve$ and \eqref{eq:check} is satisfied. Otherwise, there is some $q\in\min(S\cap \supp(\bx))$ with $q<p_j$. Furthermore, there is a marked element $a'<q$ and since $q\in S$ this element satisfies $\la(a')=0=\la(a)$. Thus, the chain $a'<q<p_j\prec\cdots\prec p_k\prec b$ together with $x_{p_q}\ge \ve$ yields
    \begin{equation*}
        \underbrace{x_{p_1} + \cdots + x_{p_{j-1}}}_0 + x_{p_j} + \cdots +x_{p_r} + \ve \le
        x_{p_q} + x_{p_j} + \cdots +x_{p_r} \le
        \la(b)-\la(a).
    \end{equation*}
    We conclude that the difference in \eqref{eq:given} is at least $\ve$ and hence \eqref{eq:check} still holds.

    Thus, we have shown that $\bz\in\CC(P,\la-\ve \om_1)$ and thus may conclude that $\CC(P,\la) = \CC(P,\la-\ve \om_1)+\CC(P,\ve\om_1)$.

    Effectively, we replaced $\la$ with $c_0=0$ by a marking $\la'=\la-\ve\om_1$ with $c_0=0$ and $c_i=c_i-\ve$ for $i\ge 1$. Repeating this procedure yields $\la-c_1\om_1$ after finitely many steps, since in each case one of the following happens:
    \begin{enumerate}
    \item $S\cap \supp(\bx)$ is empty and we reach $\la-c_1\om_1$ immediately,
    \item $S\cap \supp(\bx)$ is non-empty and $\ve=c_1$, so that we also reach $\la-c_1\om_1$,
    \item $S\cap \supp(\bx)$ is non-empty and $\ve=\tilde\ve$, so that at least one coordinate in $\min(S\cap\supp(\bx))$ is non-zero for $\bx$ but zero for $\bz$.
    \end{enumerate}
    Since $S$ is finite, the third situation can only occur finitely many times.
    Hence, we conclude that
    \begin{equation*}
        \CC(P,\la) = \CC(P,\la-c_1 \om_1) + c_1\, \CC(P,\om_1).
    \end{equation*}
    Furthermore, when $\la$ is integral, we may choose $\ve=1$ whenever $S\cap\supp(\bx)$ is non-empty to obtain
    \begin{equation*}
        \CC^\Z(P,\la) = \CC^\Z(P,\la-c_1 \om_1) + c_1\, \CC^\Z(P,\om_1).
    \end{equation*}

    The statement of the proposition now follows by induction as in the previous proof.
\end{proof}

\begin{lemma}
    \label{lem:CO-fibered}
    Let $P$ be a poset with a decomposition $P=P^*\sqcup C\sqcup O$ into marked, chain and order elements.
    For $\bx = (\la, \bx_C, \bx_O) \in \R^P$ we have $\bx\in\OO_{C,O}(P,\la)$ if and only if
    $\bx_O\in \OO(P\setminus C, \la)$ and $\bx_C\in\CC(P,\la\sqcup\bx_O)$, where $\la\sqcup\bx_O$ is the map $P^*\sqcup O\to\R$ that restricts to $\la$ on $P^*$ and to $\bx_O$ on $O$.

\end{lemma}
\begin{proof}
    This is an immediate consequence of the definition of $\OO_{C,O}(P,\la)$.
\end{proof}

\begin{theorem}
    \label{Thm:CO-decomp}
    The marked chain-order polyhedron $\OO_{C,O}(P,\la)$ decomposes as the weighted Minkowski sum
    \begin{align*}
        \OO_{C,O}(P,\la) = c_0 \, \OO_{C,O}(P,\om_0) + (c_1-c_0) \, \OO_{C,O}(P,\om_1) + \cdots + (c_k-c_{k-1}) \, \OO_{C,O}(P,\om_k).
    \end{align*}
    Furthermore, when $\la$ is integral, we have
    \begin{align*}
        \OO_{C,O}^\Z(P,\la) = c_0 \, \OO_{C,O}^\Z(P,\om_0) + (c_1-c_0) \, \OO_{C,O}^\Z(P,\om_1) + \cdots + (c_k-c_{k-1}) \, \OO_{C,O}^\Z(P,\om_k).
    \end{align*}
\end{theorem}
\begin{proof}
    We apply \Cref{lem:CO-fibered} to reduce the claim to \Cref{prop:order-decomp} and \Cref{prop:chain-decomp}.
    We start by showing that
    \[\OO_{C,O}(P,\la)=\OO_{C,O}(P,\la-c_0\om_0)+c_0\,\OO_{C,O}(P,\om_0).\]
    Let $\bx=(\la,\bx_C,\bx_O)\in \OO_{C,O}(P,\la)$ and $\widehat\la=\la\sqcup \bx_O$. Denote by $\widehat c_0<\widehat c_1<\cdots<\widehat c_l$ the elements of $\widehat\la(P^*\sqcup O)$ and note that $\widehat c_0=c_0$ since $P^*$ contains all minimal elements of $P$.
    By \Cref{lem:CO-fibered}, we have $\bx_O\in\OO(P\setminus C,\la)$, $\bx_C\in\CC(P,\widehat\la)$ and $\bx$ decomposes as $\bx=\bz+c_0\,\by$ with $\bz=(\la-c_0\om_0,\bz_C,\bz_O)$ and $\by=(\om_0,\by_C,\by_O)$ satisfying
    \begin{align*}
        \begin{cases}
            \bz_C
                \in\CC(P,\widehat\la-c_0\widehat\om_0), \\
            \bz_O
                \in\OO(P\setminus C,\la-c_0\om_0), \\
            \by_C
                \in\CC(P,\widehat\om_0),\quad\text{and} \\
            \by_O
                \in\OO(P\setminus C,\om_0).
        \end{cases}
    \end{align*}
    As in the the proofs of \Cref{prop:order-decomp} and \Cref{prop:chain-decomp}, we may choose $\by_O=\mathbf 1$, $\bz_O=\bx_O-c_0\mathbf 1$, $\by_C=\mathbf 0$ and $\bz_C=\bx_C$.
    We claim that
    \begin{align*}
        \widehat\la-c_0\widehat\om_0 &= (\la-c_0\om_0)\sqcup \bz_O \quad\text{and}\\
        \widehat\om_0&=\om_0\sqcup \by_O,
    \end{align*}
    so that $\bz\in\OO_{C,O}(P,\la-c_0\om_0)$ and $\by\in\OO_{C,O}(P,\om_0)$.
    Note that $\om_0\sqcup \by_O=\mathbf 1\sqcup \mathbf 1=\mathbf 1=\widehat\om_0$ and hence
    \begin{align*}
        (\la-c_0\om_0)\sqcup \bz_O = (\la-c_0\mathbf 1)\sqcup(\bx_O-c_0\mathbf 1)
        = (\la\sqcup \bx_O)-c_0 \mathbf 1
        = \widehat\la - c_0 \widehat\om_0.
    \end{align*}

    Thus, we may assume that $c_0$ and $\widehat c_0$ are zero and proceed by showing that
    \[\OO_{C,O}(P,\la)=\OO_{C,O}(P,\la-c_1 \om_1)+c_1\,\OO_{C,O}(P,\om_1).\]
    As before, let $\bx=(\la,\bx_C,\bx_O)\in \OO_{C,O}(P,\la)$ and $\widehat\la=\la\sqcup \bx_O$, then $\bx_O\in\OO(P\setminus C,\la)$ and $\bx_C\in\CC(P,\widehat\la)$.
    Choose $0<\ve\le \widehat c_1$ as in the proof of \Cref{prop:chain-decomp} for $\CC(P,\widehat\la)$ to obtain a decomposition $\bx_C=\bz_C+\by_C$ with $\bz_C\in\CC(P,\widehat\la-\ve\widehat\om_1)$ and $\by_C\in\CC(P,\ve\widehat\om_1)$.
    Taking the same $\ve$ in the proof of \Cref{prop:order-decomp} for $\OO(P\setminus C,\la)$, we obtain a decomposition $\bx_O=\bz_O+\by_O$ with $\bz_O\in\OO(P\setminus C,\la-\ve\om_1)$ and $\by_O\in\OO(P\setminus C,\ve\om_1)$, where $(\by_O)_p = \min\{x_p, \ve\}$ and $(\bz_O)_p=\max\{0,x_p-\ve\}$.
    In analogy to the previous step, we only need to show that
    \begin{align*}
        \widehat\la-\ve\widehat\om_1 &= (\la-\ve\om_1)\sqcup \bz_O \quad\text{and}\\
        \ve\widehat\om_1&=\ve\om_1\sqcup \by_O.
    \end{align*}
    Since $0<\ve\le\widehat c_1\le c_1$ and $\widehat c_1$ is the smallest positive value of $\widehat\la=\la\sqcup \bx_O$, we have
    \begin{align*}
        \ve\widehat\om_1(p) = \min\{x_p, \ve\} = \ve\om_1 \sqcup \by_O.
    \end{align*}
    It follows that
    \begin{align*}
        \widehat\la-\ve\widehat\om_1 =
        (\la\sqcup \bx_O)-(\ve\om_1\sqcup \by_O)
        = (\la-\ve\om_1)\sqcup(\bx_O-\by_O)
        = (\la-\ve\om_1)\sqcup \bz_O.
    \end{align*}
    Hence, we have shown that
    \[\OO_{C,O}(P,\la)=\OO_{C,O}(P,\la-\ve \om_1)+\OO_{C,O}(P,\ve\om_1)\]
    and the rest of the proof is an induction as in the proof of \Cref{prop:chain-decomp}, where we may choose $\ve=1$ in the integral case to obtain integral decompositions.
\end{proof}

\subsection{Reinterpretation: The Cone of Markings and Chain-Order Cones}

In this section we give a reinterpretation of \Cref{Thm:CO-decomp} using a subdivision of the cone of all markings associated to a poset $P$ with a set of marked elements $P^*$.

\subsubsection{Subdivision of Order Cones}
Let $P$ be a finite poset. The order cone of $P$ is defined by:
$$\mathcal{L}(P)=\{f\colon P\rightarrow\mathbb{R}\mid f\text{ is order preserving}\}\subseteq\mathbb{R}^P.$$
We consider the following set of chains of order ideals in $P$:
$$\mathfrak{I}(P)=\{ ( I_0,I_1,\cdots,I_{k-1} )\mid \emptyset\neq I_0 \subsetneq \cdots \subsetneq I_{k-1}\neq P \text{ is a chain of order ideals in }P\}.$$
For convenience we set $I_{-1}=\emptyset$ and $I_k=P$. The set $\mathfrak{I}(P)$ admits a poset structure given by coarsening: for $\mathcal{I},\mathcal{J}\in\mathfrak{I}(P)$, $\mathcal{I}\leq \mathcal{J}$ if and only if $\mathcal{I}$ is obtained by deleting some of the order ideals from $\mathcal{J}$ (in this case we say $\mathcal{I}$ is a coarsening of $\mathcal{J}$).

We define a map $\beta\colon\mathcal{L}(P)\rightarrow \mathfrak{I}(P)$ sending an order-preserving map $f\colon P\rightarrow \mathbb{R}$ to the chain of order ideals 
$$\mathcal{I}_f=(f^{-1}(\R_{\le c_0}),\,f^{-1}(\R_{\le c_1}),\,\dots,\,f^{-1}(\R_{\le c_{k-1}}))\in\mathfrak{I}(P)$$
where $f(P)=\{c_0<c_1<\cdots<c_k\}$.

For $\mathcal{I}\in\mathfrak{I}(P)$ we define a closed subcone (in the real topology) $\mathcal{L}(P,\mathcal{I})=\overline{\beta^{-1}(\mathcal{I})}$. It has the following description: for $\mathcal{I}=(I_0,I_1,\dots,I_{k-1})\in\mathfrak{I}(P)$,
$$\mathcal{L}(P,\mathcal{I})=\{f\in\mathcal{L}(P)\mid f\text{ is constant on }I_j\backslash I_{j-1}\text{ and }f(I_0\backslash I_{-1})\leq\cdots\leq f(I_k\backslash I_{k-1})\}.$$

The following statements hold (see \cite{Sta86,JS14}):
\begin{enumerate}
\item The set of cones $\{\mathcal{L}(P,\mathcal{I})\mid \mathcal{I}\in\mathfrak{I}(P)\}$ forms a polyhedral subdivision of the cone $\mathcal{L}(P)$. 
\item For $f\in\mathcal{L}(P)$, $f\in\mathcal{L}(P,\mathcal{I})$ if and only if $\mathcal{I}_f\leq\mathcal{I}$.
\item For $f\in\mathcal{L}(P)$, $f\in\mathrm{relint}(\mathcal{L}(P,\mathcal{I}))$ if and only if $\mathcal{I}_f=\mathcal{I}$.
\end{enumerate}

For a chain $\II=(I_0,I_1,\dots,I_{k-1})\in\mathfrak{I}(P)$, note that the lineality space of $\LL(P,\II)$ consists of all constant maps $P\to\R$, so that $\LL(P,\II)/\R\mathbf 1$ is a pointed polyhedral cone in $\R/\R\mathbf 1$ that we refer to as $\overline\LL(P,\II)$.
Denoting the indicator function of $P\setminus I_{j-1}$ by $\phi_j$, we see that $\overline\LL(P,\II)$ is a unimodular simplicial cone with ray generators $[\phi_1],\dots,[\phi_k]$.

\subsubsection{Chain-Order Cones}
When $P$ is a poset with a subset $P^*$ of marked elements, the construction in the previous section, when applied to $P^*$, yields a subdivision of the cone of all order-preserving markings $\LL(P^*)$ where the cells $\overline\LL(P^*,\II)$ are unimodular simplicial cones and the ray generators $[\phi_j]$ are elementary markings.

Letting $\lambda$ vary over all of $\LL(P^*)$, the marked chain-order polyhedra $\OO_{C,O}(P,\lambda)$ form a cone:
\begin{definition}
    \label{def:CO-cone}
    Let $P = P^*\sqcup C\sqcup O$ be a partition of a poset $P$ with $\min(P)\subseteq P^*$. The \emph{chain-order cone} $\OO_{C,O}(P)\subseteq\R^P$ is the set of all $\bx=(x_p)_{p\in P}\in\R^P$ satisfying the following conditions:
    \begin{enumerate}
        \item for $p\in C$, $x_p\geq 0$;
        \item for each saturated chain $a\prec p_1\prec\cdots\prec p_r\prec b$ with $a,b\in P^*\sqcup O$, $p_i\in C$, $r\geq 0$, we have
            \[x_{p_1}+\cdots+x_{p_r}\leq x_b-x_a.\]
    \end{enumerate}
\end{definition}

We let $\pi\colon\mathcal{O}_{C,O}(P)\rightarrow \mathcal{L}(P^*)$ denote the linear projection onto the coordinates corresponding to $P^*$. Then for $\lambda\in\mathcal{L}(P^*)$, the fiber $\pi^{-1}(\lambda)=\mathcal{O}_{C,O}(P,\lambda)$ is the marked chain-order polyhedron.

The polyhedral subdivision $\{\mathcal{L}(P^*,\mathcal{I})\mid \mathcal{I}\in\mathfrak{I}(P^*)\}$ induces a polyhedral subdivision 
\[\left\{\,\mathcal{O}_{C,O}(P,\mathcal{I})\coloneqq\pi^{-1}(\mathcal{L}(P^*,\mathcal{I}))\,\big|\, \mathcal{I}\in\mathfrak{I}(P^*)\right\}\]
of the chain-order cone $\mathcal{O}_{C,O}(P)$.

Since the elementary markings $\omega_j$ associated to a marking $\lambda$ are determined by $\II_\la$, we can now reformulate \Cref{Thm:CO-decomp} as follows:

\begin{theorem}
    \label{Thm:CO-cone-decomp}
Let $P$ be any poset and $P=P^*\sqcup C\sqcup O$ a partition into marked, chain and order elements. For $\mathcal{I}\in\mathfrak{I}(P^*)$ and $\lambda,\mu\in\mathcal{L}(P^*,\mathcal{I})$, 
\[\OO_{C,O}(P,\lambda+\mu)=\OO_{C,O}(P,\lambda)+\OO_{C,O}(P,\mu).\]
Furthermore, if $\lambda$ and $\mu$ are both integral, then 
\[\OO_{C,O}^\Z(P,\lambda+\mu)=\OO_{C,O}^\Z(P,\lambda)+\OO_{C,O}^\Z(P,\mu).\]
\end{theorem}

\begin{proof}
    For the chains of order ideals, we have $\II_\la,\II_\mu\le \II_{\la+\mu}\le\II$, since $\II_\la$ and $\II_\mu$ have $\II$ as a common refinement.
    In other words, the elementary markings $\omega_i$ appearing in the decompositions of $\lambda$, $\mu$ and $\lambda+\mu$ form a subset of $\{\mathbf 1,\phi_1,\dots,\phi_k\}$, where $\mathbf 1$ is the constant marking of all ones and the $\phi_j$ are the indicator functions of $P\setminus I_{j-1}$. Using $(\alpha+\beta)\, \OO_{C,O}(P,\omega_i) = \alpha \,\OO_{C,O}(P,\omega_i) + \beta \,\OO_{C,O}(P,\omega_i)$ and $(\alpha+\beta)\, \OO^\Z_{C,O}(P,\omega) = \alpha\, \OO^\Z_{C,O}(P,\omega_i) + \beta\,\OO^\Z_{C,O}(P,\omega_i)$ (for integral $\alpha,\beta\ge 0$) we may thus decompose points of $\OO_{C,O}(P,\la+\mu)$ in terms of elementary markings and then redistribute summands to obtain a sum of points of $\OO_{C,O}(P,\la)$ and $\OO_{C,O}(P,\mu)$.
\end{proof}

Note that in \Cref{ex:bad-decomp} the two markings given by $(0,1,0)$ and $(0,0,1)$ yield chains of order ideals that do not admit a common refinement.

\begin{remark}
    Let $\lambda$ be integral and $\II\in\mathfrak I(P^*)$ be a chain of order ideals in $P^*$. By \Cref{Thm:CO-decomp} we can take any lattice point $\bx=(\lambda,\bx_C,\bx_O)\in\OO^\Z_{C,O}(P,\II)$, decompose $\lambda$ into elementary markings---hence expressing it as a sum of $\mathbf 1$s and minimal ray generators $\phi_j$ of $\overline\LL(P^*,\II)$---to then decompose $\bx$ as a sum of lattice points in the polytopes $\OO_{C,O}(P,\phi_j)$ and $\OO_{C,O}(P,\mathbf 1)$.

    We may rephrase this to obtain a generating set of the semigroup algebra associated to $\OO^\Z_{C,O}(P,\II)$.
    Denote by $\C[\OO^\Z_{C,O}(P,\II)]$ the $\C$-algebra with basis elements $\chi^\bx$ for all $\bx\in\OO^\Z_{C,O}(P,\II)$ and multiplication defined by $\chi^\bx \cdot \chi^\by \coloneqq \chi^{\bx+\by}$.
    For $\la\in\LL(P,\II)$ let $\C[\OO^\Z_{C,O}(P,\II)]_\lambda$ be the subspace spanned by the $\chi^\bx$ with $\bx\in\OO^\Z_{C,O}(P,\la)$, then
    \begin{equation*}
        \C[\OO^\Z_{C,O}(P,\II)] = \bigoplus_{\lambda\in\LL(P,\II)} \C[\OO^\Z_{C,O}(P,\II)]_\lambda
    \end{equation*}
    and the algebra is generated by the $\chi^\bx$ where $\bx\in\OO^\Z_{C,O}(P,\lambda)$ and $\lambda\in\{\mathbf 1, \phi_1,\dots,\phi_k\}$.
\end{remark}

\section{Reflexivity of Marked Chain-Order Polytopes}\label{Sec:3}

In this section, we will assume that the marked poset $(P,\la)$ satisfies $\max(P)\cup\min(P)\subseteq P^*$ and the marking $\la$ is integral.

\subsection{Unique Interior Lattice Points}

Let $P$ be a ranked poset with rank function $r$ and $(P,\la^r)$ be the marking arising from the rank function and a choice of marked elements. Then the point $\mathbf{r}= (r_p)_{p\in P}$ defined by $r_p = r(p)$ is contained in $\OO(P,\la^r)$.

\begin{proposition}\label{Prop:UniqueLatticePoint}
    The point $\br$ is the unique interior lattice point of $\OO(P,\la^r)\subseteq\R^{\tilde P}$.
\end{proposition}
\begin{proof}
By \Cref{Prop:Facet}, the point $\mathbf{r}$ is an interior lattice point in $\OO(P,\la^r)$. 
\par
For the uniqueness, let $\mathbf{r}'=(r_p')\neq \mathbf{r}$ be another interior lattice point in $\OO(P,\la^r)$. Let $p\in\tilde{P}$ be arbitrary and consider a saturated chain between marked elements containing $p$, say 
\[a\prec p_1\prec\cdots\prec p_k=p\prec\cdots\prec p_s\prec b\]
where $a,b\in P^*$ and $p_1,\cdots,p_s\in\tilde{P}$. Note that for each covering relation $p_i\prec p_{i+1}$ we must have $r'_{p_i} < r'_{p_{i+1}}$. Since $\br'$ is integral and $a, b$ are marked with their rank, this is only possible for $r'_{p_i} = r(a)+i = r_{p_i}$.
\end{proof}

\begin{corollary}\label{Cor:UniqueLatticePoint}
For any partition $\tilde{P}=C\sqcup O$, the marked poset polytope $\OO_{C,O}(P,\la^r)$ has a unique interior lattice point. 
\end{corollary}

\begin{proof}
Let $Q= \OO(P,\la)$. We apply the Ehrhart-Macdonald reciprocity to $Q$, by the above proposition, $L_Q(-1)=(-1)^{\dim(Q)}$. By \Cref{Prop:Ehrhart}, the value of the Ehrhart polynomial of $\OO_{C,O}(P,\la^r)$ at $-1$ is $(-1)^{\dim(Q)}$.
\end{proof}

\begin{remark}
Another proof of this corollary without using the Ehrhart theory can be executed using the transfer map \cite{FFLP17}: it suffices to notice that the transfer map preserves not only the lattice points, but also the boundary of the polytope by continuity.
\end{remark}

\subsection{Reflexivity}

Let $Q\subseteq\R^n$ be a lattice polytope with a unique interior lattice point $\bu$. We denote $\overline{Q}= Q-\bu$ be the canonically translated polytope.

By \Cref{Cor:UniqueLatticePoint}, for any partition $\tilde{P}=C\sqcup O$, the translated polytope $\overline{\OO}_{C,O}(P,\la^r)$ is well-defined, and contains $0$ as its unique interior lattice point.
We continue by characterizing the reflexive polytopes among marked chain-order polytopes.

\begin{theorem}\label{Thm:main}
    Let $(P,\la)$ be a regular marked poset. The following statements are equivalent:
    \begin{enumerate}[label=\roman*)]
        \item $P$ is a ranked poset and $\la=\la^r$ is a marking arising from a rank function;
        \item for any partition $\tilde{P}=C\sqcup O$, the polytope $\OO_{C,O}(P,\la)$ has a unique interior lattice point and $\overline{\OO}_{C,O}(P,\la)$ is reflexive.
    \end{enumerate}
\end{theorem}

\begin{proof}
    We start by showing that i) implies ii). Applying the transfer map we see that the unique interior lattice point of $\OO_{C,O}(P,\la^r)$ is given by $x_p = r(p)$ for $p\in O\sqcup P^*$ and $x_p=1$ for $p\in C$. Hence, after shifting this point to the origin, the defining inequalities of $Q=\overline\OO_{C,O}(P,\la^r)$ are:
    \begin{enumerate}
        \item for $p\in C$: $x_p\ge -1$,
        \item for each saturated chain $a\prec p_1\prec\cdots\prec p_r\prec b$ with $a,b\in P^*\sqcup O$, $p_i\in C$, $r\geq 0$:
            \[x_{p_1}+\cdots+x_{p_r}\leq x_b-x_a+1.\]
    \end{enumerate}

    The polarity theorem (\Cref{Prop:Polarity}) can be then applied to conclude that the vertices of $Q^\circ$ have integral coordinates and hence $\overline\OO_{C,O}(P,\la^r)$ is reflexive.
    
    For the other direction we note that $\overline\OO_{C,O}(P,\la)$ is reflexive if and only if $\overline\OO(P,\la)$ is reflexive since the polytopes are Ehrhart equivalent and the Ehrhart polynomial determines reflexivity by \Cref{Thm:Hibi}.

    If $P$ is ranked but the marking $\la$ does not arise from a rank function, there exist $a,b\in P^*$ such that $a<b$ but $\la(b)-\la(a)$ is strictly larger than the length of a (hence of any) saturated chain between $a$ and $b$. We take such a saturated chain $a\prec p_1\prec\cdots\prec p_k\prec p_{k+1}=b$; without loss of generality we may assume that $p_1,\dots,p_k\in \tilde{P}$. Since $\OO(P,\la)$ has only one interior lattice point $\mathbf{a}$, there exists $1\leq r\leq k$ such that $a_{p_{r+1}}-a_{p_r}=\ell\geq 2$. Since $(P,\la)$ is regular, after translation, $\overline{\OO}(P,\la)$ has a facet defined by $x_{p_{r+1}}-x_{p_r}\geq \ell$: by the polarity theorem (\Cref{Prop:Polarity}), it is not reflexive.
    
    If $P$ is not ranked, consider the unique interior lattice point $\mathbf{a}$. There exist a covering $p\prec q\in P$ not both in $P^*$ such that $a_p<a_q-1$, since otherwise $\mathbf a$ would give rise to a rank function on $P$. Now the same argument as above applies.
\end{proof}

\begin{remark}
\begin{enumerate}
\item We also provide a geometric proof of the fact that $\overline\OO(P,\la^r)$ is reflexive by applying \Cref{Thm:GeomReflexive}.
The facet defining inequalities of $\OO(P,\la^r)$ are given by $x_p\le x_q$ for each covering relation $p\prec q$ in $P$.
The affine hull of such a facet is the hyperplane $x_q-x_p=0$ and its parallel through the unique interior lattice point $\br$ is given by $x_q-x_p=1$. Since there is no integer between $0$ and $1$, there are no lattice points between these two hyperplanes.
Translating $\OO(P,\la^r)$ such that $\br$ becomes the origin yields reflexivity of $\overline\OO(P,\la^r)$ by \Cref{Thm:GeomReflexive}.
\item In \cite[Lemma on p.~96]{Hibi87} Hibi showed that the order polytope is Gorenstein, i.e., it has an integral dilate that is reflexive up to translation, if and only if the defining poset is ranked. In this setting $P$ is a poset with the unique minimum $\hat 0$ and maximum $\hat 1$ being marked. The order polytope is obtained by $\lambda(\hat 0)=0$ and $\lambda(\hat 1)=1$. Dilating it by the rank of the poset yields the marked order polytope given by $\lambda=\lambda^r$ so that $\overline{\OO}(P,\lambda^r)$ is reflexive.
\end{enumerate}
\end{remark}

\subsection{Counter-Examples and Remark}

We illustrate the obstructions in \Cref{Thm:main} in two examples.

\begin{example}
    If $P$ is a ranked poset while $\la\colon P^*\to\Z$ is not a rank function on $P$, the polytope $\OO(P,\la)$ might not be reflexive despite having only one interior lattice point.
    We consider the following marked poset:
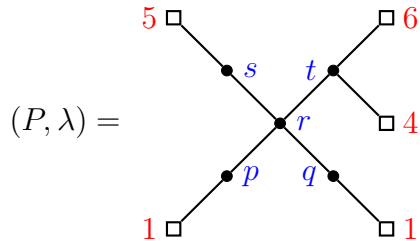
\begin{center}
$(P,\la)=$ \begin{tikzpicture}[baseline={([yshift=-.5ex]current bounding box.center)},scale=0.7]
            \path (-2,-2) node[posetelmm] (0) {} node[left=2pt,marking] {$1$};
            \path (2,-2) node[posetelmm] (1) {} node[right=2pt,marking] {$1$};
            \path (-2,2) node[posetelmm] (2) {} node[left=2pt,marking] {$5$};
            \path (2,2) node[posetelmm] (3) {} node[right=2pt,marking] {$6$};
            \path (2,0) node[posetelmm] (4) {} node[right=2pt,marking] {$4$};
            \path (-1,1) node[posetelm] (x) {} node[right=2pt,elmname] {$s$};
            \path (0,0) node[posetelm] (y) {} node[right=2pt,elmname] {$r$};
            \path (1,1) node[posetelm] (z) {} node[left=2pt,elmname] {$t$};
            \path (1,-1) node[posetelm] (a) {} node[left=2pt,elmname] {$q$};
            \path (-1,-1) node[posetelm] (b) {} node[right=2pt,elmname] {$p$};
            \draw[covrel]
                  (0) -- (b) --  (y) --  (z) --  (3)
                  (z) --  (4)
                  (1) -- (a) -- (y) --  (x) -- (2);
\end{tikzpicture}
\end{center}
The marked order polytope $\OO(P,\la)$ has the unique interior lattice point
\[(x_p,x_q,x_r,x_s,x_t)=(2,2,3,4,5).\]
But the dual polytope of the translated polytope $\overline{\OO}(P,\la)$ is not a lattice polytope.
\end{example}

\begin{example}
We consider the following poset
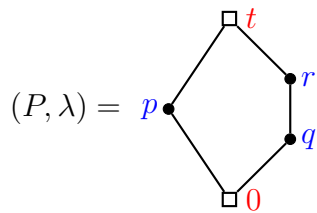
\begin{center}
$(P,\la)=$ \begin{tikzpicture}[baseline={([yshift=-.5ex]current bounding box.center)},scale=0.8]
            \path (0,0) node[posetelmm] (0) {} node[right=2pt,marking] {$0$};
            \path (-1,1.5) node[posetelm] (x) {} node[left,elmname] {$p$};
            \path (1,1) node[posetelm] (y) {} node[right,elmname] {$q$};
            \path (1,2) node[posetelm] (z) {} node[right,elmname] {$r$};
            \path (0,3) node[posetelmm] (1) {} node[right=2pt,marking] {$t$};
            \draw[covrel]
                  (0) -- (x) --  (1)
                  (0) -- (y) -- (z) --  (1);
\end{tikzpicture}
\end{center}
where the maximal and minimal elements are marked by $t\in\N$ and $0$ respectively, $\tilde{P}=\{p,q,r\}$. We show that the marked order polytope $\OO(P,\la)$ can not have a unique interior lattice point.

When $t\le 2$, there are no integers $x_q, x_r$ such that $0<x_q<x_r<t$. When $t\ge 3$, the two points with $x_q=1$, $x_r=2$ and $x_p\in\{1,2\}$ are both interior lattice points of $\OO(P,\la)$.
\end{example}

\begin{remark}
    Let $G=\SL_{n+1}$ or $\mathrm{Sp}_{2n}$, $B$ be a Borel subgroup in $G$, and $G/B$ be the complete flag variety embedded in $\mathbb{P}(V(2\rho))$ (embedding using the anti-canonical bundle on $G/B$) where $2\rho$ is the sum of positive roots in $G$. As shown in \cite{GL96, Cal02} and \cite{FFL17}, there exist flat toric degenerations of $G/B$ to the toric varieties associated to the marked order polytopes and marked chain polytopes associated to Gelfand-Tsetlin posets where the marking is given by $2\rho$ (see \cite{ABS11, Fou16} for the definition of the posets). By \Cref{Thm:main}, these toric varieties are Gorenstein and Fano. The same follows from more general results in \cite{Ste} on the reflexivity of Newton-Okounkov bodies arising from flag varieties.
\end{remark}

\subsection*{Acknowledgements}

Part of the work was carried out during a research visit of X.F.\ to University of Hannover. He would like to thank University of Hannover for the hospitality. We would like to thank the anonymous reviewers, whose detailed comments helped to improve this work.

\printbibliography

\end{document}